\documentclass[11pt,a4paper]{article}
\usepackage{amsmath,amsthm,amssymb,amsfonts}
\usepackage{hyperref}
\usepackage{enumitem}

\newcommand{\XIS}{{\mathfrak{X}}}

\newcommand{\SO}{{\mathrm{SO}}}

\newcommand{\rr}{\rightarrow}
\newcommand{\lrr}{\longrightarrow}

\newcommand{\call}{{\cal L}}             %
\newcommand{\calri}{{{\cal R}^\xi}}             %
\newcommand{\calu}{{\cal U}}             %

\newcommand{\na}{{\nabla}}
\newcommand{\tr}[1]{{\mathrm{tr}}{#1}}

\newcommand{\End}[1]{{\mathrm{End}}\,{#1}}

\newcommand{\dx}{{\mathrm{d}}}

\newcommand{\estrela}{{\boldsymbol{\star}}}

\newcommand{\vol}{{\mathrm{vol}}}
\newcommand{\ric}{{\mathrm{Ric}}}

\newtheorem{teo}{Theorem}[section]

\newtheorem{coro}{Corollary}[section]
\newtheorem{prop}{Proposition}[section]

\def\cyclic{\mathop{\kern0.9ex{{+}
\kern-2.2ex\raise-.28ex\hbox{\Large\hbox{$\circlearrowright$}}}}\limits}

\title{New Hsiung-Minkowski identities}

\author{Rui Albuquerque}

\begin{document}


\maketitle


\begin{abstract}

We find the first three most general Minkowski or Hsiung-Minkowski 
identities relating the total mean curvatures $H_i$, of degrees 
$i=1,2,3$, of a closed hypersurface $N$ immersed in a given orientable 
Riemannian manifold $M$ endowed with any given vector field $P$. Then we specialise the three identities to the case when $P$ is a position vector field. We further obtain that the classical Minkowski identity is natural to all Riemannian manifolds and, moreover, that a corresponding 1st degree Hsiung-Minkowski identity holds true for all Einstein manifolds. We apply the result to hypersurfaces with constant $H_1,H_2$.

\end{abstract}


\ 
\vspace*{3mm}
\vspace*{2mm}

\vspace*{10mm}

\markright{\sl\hfill  R. Albuquerque \hfill}

\setcounter{section}{1}

\begin{center}
 \textbf{1. Introduction}
\end{center}
\setcounter{section}{1}

Let $M$ denote an orientable Riemannian manifold of class $\mathrm{C}^2$ and
dimension $n+1$. Let ${P}\in\XIS_\calu$ be a \textit{position vector 
field} of $M$ defined on an open subset $\calu\subset M$. This is, there 
exists a function $f\in\mathrm{C}^1_\calu$ such that, for any 
$X\in\XIS_\calu$,
\begin{equation} \label{def_positionvectorfield}
 \na_XP=f X.
\end{equation}

Let $N$ be a closed orientable isometrically immersed hypersurface of $M$ of 
class $\mathrm{C}^2$, contained in $\calu$. Let $\nu$ denote one of the two 
unit normals to $N$ in $M$ and let $H_i$ denote the $i$th mean curvature of 
$N$ with respect to $\nu$. As usual, $H_0=1$. Now we may recall the Hsiung-Minkowski identities: if $M$ has constant sectional curvature, then, for $0\leq i\leq n-1$,
\begin{equation} \label{HMI-csc}
 \int_N(f H_i -\langle {P},\nu\rangle H_{i+1})\,\vol_N=0 .
\end{equation}
This identity was found by Hsiung in \cite{Hsiung1} in the case of 
Euclidean space, with position vector field $P_x=x\in T_xM$, hence with 
$f=1$, generalizing the same result of Minkowski for $i=0$.

The Hsiung-Minkowski identities were further improved by Y.~Katsurada, in 
\cite{Katsurada}, who found certain \textit{generalized} Minkowski 
identities for a conformal-Killing vector field, in place of $P$, still on 
the ambient of a constant sectional curvature $M$ contained in Euclidean 
space. The true formula and meaning of a position vector field there remains 
a bit obscure. Nevertheless, since $P$ is also conformal-Killing, there is 
some overlap with results presented here and today. B.~Chen and K.~Yano in 
\cite{ChenYano} also gave the most general integral formulas, for closed 
submanifolds of higher codimension in Euclidean space. Regarding the other 
space forms, P.~Guan and J.~Li \cite{GuanLi} and C.~Guidi and 
V.~Martino \cite{GuidiMartino} gave independent proofs of \eqref{HMI-csc} 
recurring to Newton's identities for symmetric polynomials. In the same trend 
the new formulas of K.~Kwong \cite{Kwong,KwongLeePyo} have appeared. We 
refer 
the reader to \cite{Alb2019} for a new independent proof of \eqref{HMI-csc}. 
Only recently we found \cite{Bivens}, most probably the first to give the 
definition of position vector field and to give a proof of \eqref{HMI-csc} 
after the singular approach of \cite{Katsurada}.

In this article, in Theorem~\ref{teo_vgHMI}, we deduce the most general 
integral identities of the kind of Hsiung-Minkowski for the 
cases $i=0,1,2$. Indeed generalized, in as much as they are concerned with 
any vector field and any oriented Riemannian manifold $M$.

Given a position vector field $P$, we show in Theorem~\ref{teo_HMI} that
\begin{equation}
 \int_N(f-\langle {P},\nu\rangle H_{1})\,\vol=0
\end{equation}
and
\begin{equation}
 \int_N\Bigl(2\binom{n}{2}f H_1-2\binom{n}{2}\langle P,\nu\rangle 
H_2-\ric(P,\nu)+\langle P,\nu\rangle\ric(\nu,\nu)\Bigr)\,\vol=0
\end{equation}
and
\begin{equation}
 \begin{split}
& \int_N\Bigl(
 \bigl(\langle P,\nu\rangle\ric(\nu,\nu)-\ric(\nu,P)\bigr)nH_1
 +3\binom{n}{3}fH_2-3\binom{n}{3}\langle P,\nu\rangle H_3+ \Bigr. \\
 &\qquad\qquad \Bigl. +\langle 
P,\nu\rangle\tr{(R(\nu,\na_\cdot\nu)\nu)}+\ric(\na_P\nu,\nu)
 -\tr{(R(P,\na_\cdot\nu)\nu)}\Bigr)\,\vol =0.
  \end{split}
\end{equation}

We then find that \eqref{HMI-csc} is valid for every Riemannian manifold in 
case $i=0$, valid for every Einstein manifold in cases $i=0,1$, and, 
still as a corollary, valid for constant sectional curvature in cases 
$i=0,1,2$. Let us remark that every warped product metric admits a position 
vector field. We end the article with some applications to the theory of 
isoparametric hypersurfaces in a given Einstein manifold.

This article brings to light a new application of a fundamental differential 
system of Riemannian geometry introduced in \cite{Alb2012}. Given our 
self-contained methods in the treatment of the $i$th-mean curvatures and 
beyond, we again recommend the reading of \cite{Alb2019}.

\vspace{4mm}
\begin{center}
 \textbf{2. Recalling the fundamental differential system}
\end{center}
\setcounter{section}{2}

Our framework for the proof of the generalized integral identities is that 
of the unit tangent sphere bundle $\pi:SM\lrr M$ together with its exterior 
differential system of invariant $n$-forms on the $(2n+1)$-dimensional 
manifold $SM$. Let us start by recalling such fundamental differential 
system. We use the notation $\langle\ ,\ \rangle$ for the metric.

We take the tangent bundle $\pi:TM\lrr M$ with the canonical Sasaki metric 
and $\SO(n+1)$ structure and recall the \textit{mirror} map $B\in\End{TTM}$ 
restricted to $SM$, see \cite{Alb2012,Alb2019}. The map is characterized by 
$B\circ B=0$ and $B\pi^*X=\pi^\estrela X$, for any $X$ tangent to $M$. The 
notation is the usual: $\pi^*$ for the horizontal lift, to the pullback 
tangent bundle identified via $\dx\pi$ with horizontals in 
$TTM=\pi^*TM\oplus\pi^\estrela TM$, and $\pi^\estrela$ for the vertical 
lift to another copy of the pullback tangent bundle, identified with 
verticals $\pi^\estrela TM=\ker\dx\pi$ only due to the differential 
structure.

The tautological or canonical vector field $\xi$, that which is defined as 
the vertical $\xi_u=u\in TTM$, $\forall u\in TM$, yields a unit horizontal 
vector field $e_0$ over $SM$ and a dual 1-form $\theta=e^0$. This defines a 
well-known contact structure. Moreover, $e_0$ is the restriction to $SM$ of 
the geodesic flow vector field. Its mirror $Be_0=\xi$ is the canonical vector 
field on $TM$. Hence we have a reduction of the structural group of the 
Riemannian manifold $SM$ to $\SO(n)$.

We further recall that a frame $e_0,e_1,\ldots,e_n,e_{n+1},\ldots,e_{2n}$ is 
said to be an \textit{adapted frame} on $SM$ if: it is orthonormal, the first 
$n+1$ vectors are horizontal, the remaining are vertical and each of the 
$e_{i+n}$ is the mirror of $e_i$.

Now let us recall the differential system of $M$ found in \cite{Alb2012}. 
This is a collection of global $n$-forms $\alpha_i$ on $SM$ defined as 
follows. First we let $\alpha_n=\xi\lrcorner\pi^\estrela\vol$ be the volume 
form of the fibres. Then for each $i=0,\ldots,n$ we let  
\begin{equation} \label{definition_alphai}
  \alpha_i =\frac{1}{i!(n-i)!}\,\alpha_n\circ(B^{n-i}\wedge1^i) .
\end{equation}
$\circ$ denotes an alternating operator which is particularly efficient 
under differentiation: Leibniz rule applies in the obvious way, with no signs 
attached. It is convenient to recall here, for any manifold, any 1-forms 
$\eta_1,\ldots\eta_p$ and any endomorphisms $B_1,\ldots,B_p$, the result that
\begin{equation}
 (\eta_1\wedge\cdots\wedge\eta_p)\circ(B_1\wedge\cdots\wedge B_p) 
 =\sum_{\sigma\in\mathrm{Sym}_p} 
 \eta_1\circ B_{\sigma_1}\wedge\cdots\wedge \eta_p\circ B_{\sigma_p} ,
\end{equation}
so the reader may know precisely what we are referring.  Notice 
$\alpha_n=e^{n+1}\wedge\cdots\wedge e^{n+n}$ and we may say $i$ is the number 
of vertical directions in each summand of the $n$-form $\alpha_i$. For 
convenience one also defines 
$\alpha_{-1}=\alpha_{n+1}=0$.

Given an orientable $\mathrm{C}^2$ immersed hypersurface 
$\iota:N\hookrightarrow M$ with the induced metric, let us recall the second 
fundamental form is defined by $A=\na\nu$, where $\nu$ is one of the two 
unit-normals to $N$. Then, for $0\leq i\leq n$, the $i$th-mean curvature 
$H_i$ is defined by $\binom{n}{i}H_i$ being the elementary symmetric 
polynomial of degree $i$ on the eigenvalues $a_1,\ldots,a_n$ of $A$, the 
so-called principal curvatures of $N$. In other words, $H_0=1$ and 
$\binom{n}{i}H_i=\sum_{1\leq j_1<\cdots<j_i\leq n}a_{j_1}\cdots a_{j_i}$.

Now we have a canonical lift $\hat{\iota}:N\hookrightarrow SM$ of $\iota$ to 
$SM$, an immersion, defined by $\hat{\iota}(x)=\nu_{\iota(x)}$. Hence
\begin{equation} \label{fundapullbacktheta}
 \hat{\iota}^*\theta=0
\end{equation}
and, recurring to \cite{Alb2012},
\begin{equation}  \label{fundapullbackalphai}
 \hat{\iota}^*{\alpha_i}=\binom{n}{i}H_i\,\vol_N .
\end{equation}

The following structural identities on the unit tangent sphere bundle of $M$ are found in 
\cite[Theorem 2.4]{Alb2012}: $\forall 0\leq i\leq n$,
\begin{equation}\label{derivadasalpha_i}
\dx\alpha_i=(i+1)\,\theta\wedge\alpha_{i+1}+\calri\alpha_i 
\end{equation}
where
\begin{equation}  \label{calriinabase}
 \calri\alpha_i = \sum_{0\leq j<q\leq n}\sum_{p=0}^nR_{p0jq}\,e^{jq}\wedge 
e_{p+n}\lrcorner\alpha_i  .
\end{equation}
and $R_{lkij} = \langle R(e_i,e_j)e_k,e_l\rangle= \langle 
\na_{e_i}\na_{e_j}e_k-\na_{e_j}\na_{e_i}e_k-\na_{[e_i,e_j]}e_k,e_l\rangle$.

Letting $r=\pi^\estrela\ric(\xi,\xi)=\sum_{j=1}^nR_{j0j0}$, a function 
on $SM$ determined by the Ricci curvature of $M$, we have also from 
\cite{Alb2012} that $\calri\alpha_0=0$ and 
$\calri\alpha_{1}=-r\,\theta\wedge\alpha_0$. In other words,
\begin{equation}\label{dalphaZeroedalphaUm}
 \dx\alpha_0=\theta\wedge\alpha_{1},\qquad\quad
 \dx\alpha_{1}=2\theta\wedge\alpha_{2}-r\,\theta\wedge\alpha_0.
\end{equation}
Notice $\theta\wedge\alpha_0=\pi^*\vol_M$. Many other general formulas like 
the above are known, but not required here.

\begin{teo} \label{teo_vgHMI}
Let $\iota:N\hookrightarrow M$ be a closed orientable isometrically immersed 
$\mathrm{C}^2$ hypersurface of $M$. Let ${P}$ be \emph{any vector field} of $M$ 
defined on a neighborhood of $\iota(N)$. Then we have that
\begin{equation}\label{vgHMzero}
 \int_N\frac{1}{n!}\hat{\iota}^*(\alpha_n\circ(\call_PB^n))-\binom{n}{1}
\langle {P},\nu\rangle H_{1}\,\vol=0
\end{equation}
and
\begin{equation}\label{vgHMum}
 \int_N\frac{1}{(n-1)!}\hat{\iota}^*(\alpha_n\circ(\call_PB^{n-1} \wedge1)) 
-\Bigl(2\binom{n}{2}\langle P,\nu\rangle H_2+\ric(P,\nu)-
\langle P,\nu\rangle\ric(\nu,\nu)\Bigr)\,\vol=0
\end{equation}
and
\begin{equation}\label{vgHMdois}
 \begin{split}
& \int_N
 \frac{1}{2(n-2)!}\hat{\iota}^*(\alpha_n\circ(\call_PB^{n-2}\wedge1^2)) 
+\Bigl(\bigl(\langle P,\nu\rangle\ric(\nu,\nu)-\ric(\nu,P)\bigr)nH_1  \Bigr. \\
 &\hspace{16mm} \Bigl.-3\binom{n}{3}\langle P,\nu\rangle H_3 +\langle 
P,\nu\rangle\tr{(R(\nu,\na_\cdot\nu)\nu)} 
+\ric(\na_P\nu,\nu) -\tr{(R(P,\na_\cdot\nu)\nu)}\Bigr)\,\vol =0.
  \end{split}
\end{equation}
\end{teo}
\begin{proof}
Regarding $\call_PB$, notice we are denoting also by $P$ the horizontal lift $\pi^*P$ of $P$ 
to $SM$. Moreover, the horizontal lift $P$ is tangent to $SM$, so we can apply Lie derivative 
to any vector bundle sections restricted to $SM$.

As with $\hat{\iota}^*\theta=0$, it will become necessary to see, what is 
quite immediate, that
\begin{equation} \label{pullbacksfundas}
  \hat{\iota}^*(P\lrcorner\theta)=\hat{\iota}^*(\theta(P))=\langle 
\nu,P\rangle 
 ,\qquad\qquad \hat{\iota}^*r=\ric(\nu,\nu) .   
\end{equation}

Now, applying the well established Leibniz rule, cf. \cite{Alb2012}, on 
\eqref{derivadasalpha_i} and noticing $P\lrcorner\alpha_n=0$, since $P$ is
horizontal, on the one hand we have
 \begin{align*}
  \call_{P}\alpha_i &= 
  \frac{1}{i!(n-i)!}\Bigl((\call_{P}\alpha_n)\circ(B^{n-i}\wedge1^i) + 
 \alpha_n\circ(\call_{P}B^{n-i}\wedge1^i) \Bigr) \\ 
  &= 
\frac{1}{i!(n-i)!}\,({P}\lrcorner\calri\alpha_n)\circ(B^{n-i}\wedge1^i)+  
\frac{1}{i!(n-i)!}\,\alpha_n\circ(\call_PB^{n-i}\wedge1^i) .
 \end{align*}
And, by Cartan's formula again, on the other hand we have
 \begin{align*}
\call_{P}\alpha_i &  =\dx(P\lrcorner\alpha_i)+ 
P\lrcorner\bigl((i+1)\theta\wedge\alpha_{i+1}+\calri\alpha_i\bigr) \\
 &=\dx(P\lrcorner\alpha_i)+ (i+1)\langle P,\nu\rangle\alpha_{i+1} 
-(i+1)\theta\wedge P\lrcorner\alpha_{i+1} +P\lrcorner\calri\alpha_i.  
 \end{align*}
 
The identities in (\ref{vgHMzero}--\ref{vgHMdois}) arise as the difference 
between the two forms above after applying the pullback by $\hat{\iota}:N\rr 
SM$ and, of course, taking integrals and applying Stokes theorem. Therefore, 
we shall firstly study the two expressions 
$(P\lrcorner\calri\alpha_{n})\circ(B^{n-i}\wedge1^i)$ and 
$P\lrcorner\calri\alpha_i$ and then take their difference in the three cases.
 
Let us start by the remark that if $T_{ij}$ is skew-symmetric in $i,j$ and 
$P=\sum P_ke_k$, then 
 \[ P\lrcorner\sum_{0\leq i<j\leq n}T_{ij}e^{i}\wedge e^j
 =\frac{1}{2}P\lrcorner\sum_{i,j=0}^n T_{ij}e^{i}\wedge e^j 
=\frac{1}{2}\sum_{i,j=0}^nT_{ij}(P_ie^{j}-P_je^i) 
=\sum_{i,j=0}^nT_{ij}P_ie^{j}  .  \]
 
Now, since the given $P$ is horizontal, $P\lrcorner e_{p+n}\lrcorner\alpha_n=0$. We notice 
$e^{j+n}\circ B=e^j$ and $e^j\circ B=0$. Henceforth 
\begin{equation}   \label{importantcomputation}
 \begin{split}
 \lefteqn{ (P\lrcorner\calri\alpha_n)\circ(B^{n-i}\wedge1^i)=  \sum_{0\leq 
j<q\leq n}\sum_{p=0}^n(R_{p0jq}\,P\lrcorner e^{jq}\wedge 
e_{p+n}\lrcorner\alpha_n)\circ(B^{n-i}\wedge1^i) } \hspace*{12mm} \\
  &= \sum_{p,j,q=0}^nR_{p0jq}P_j\,e^{q}\wedge e_{p+n}\lrcorner 
(e^{n+1}\wedge\cdots \wedge e^{n+n})  \circ(B^{n-i}\wedge1^i) \\
  &= \sum_{p,j=0}^n\Bigl( R_{p0j0}P_j\,e^{1+n}\wedge\cdots 
e^{0}\wedge\cdots\wedge e^{2n}   \\ 
 & \hspace*{20mm} +\sum_{q=1}^nR_{p0jq}P_j\,e^{1+n}\wedge\cdots 
e^{q}\wedge\cdots\wedge e^{2n}\Bigr)\circ(B^{n-i}\wedge1^i) ,
 \end{split}
\end{equation}
of course, with $e^0,e^q$ in the last line in respective positions $p$.

For $i=0$, we have $(P\lrcorner\calri\alpha_{n})\circ B^n=0$, because there  
is a factor $e^{j}$, and any $e^j\circ B=0$, and $B$ will certainly be 
everywhere after applying $\circ$. Since $\calri\alpha_0=0$ by 
\eqref{dalphaZeroedalphaUm}, the seeked differential identity is just
\[ \frac{1}{n!}\alpha_n\circ(\call_PB^n)= 
\dx(P\lrcorner\alpha_0)+\theta(P)\alpha_1-\theta\wedge P\lrcorner\alpha_1 .\]
Taking pullbacks and recalling 
(\ref{fundapullbacktheta},\ref{fundapullbackalphai},\ref{pullbacksfundas}), 
the first identity \eqref{vgHMzero} follows.

For $i=1$, since $e^0\circ B=0$ and $e^q\circ B=0$, the computation above 
continues as
\begin{equation*}  
\begin{split}
& \ \sum_{p,j=0}^n\Bigl( R_{p0j0}P_j\,e^{1+n}\wedge\cdots 
e^{0}\wedge\cdots\wedge e^{2n}    +\sum_{q=1}^nR_{p0jq}P_j\,e^{1+n}\wedge\cdots 
e^{q}\wedge\cdots\wedge e^{2n}\Bigr)\circ(B^{n-1}\wedge1) \\
& = (n-1)!\sum_{p,j=0}^nR_{p0j0}P_j\,e^{1}\wedge\cdots  
e^{0}\wedge\cdots\wedge e^{n} +(n-1)!\sum_{p,j=0}^n 
R_{p0jp}P_j\,e^{1}\wedge\cdots e^{p}\wedge\cdots\wedge e^{n}
\end{split}
\end{equation*}
and taking the pullback by $\hat{\iota}$ we find just 
$-(n-1)!\,\ric(P,\nu)\,\vol_N$. By \eqref{pullbacksfundas} and the 
second formula in \eqref{dalphaZeroedalphaUm}, we have
\[ \hat{\iota}^*(P\lrcorner\calri\alpha_1)=-\hat{\iota}^*(P\lrcorner 
r\,\theta\wedge\alpha_0)=-\ric(\nu,\nu)\langle P,\nu\rangle\,\vol_N .\]
Putting the terms in equation as explained earlier, we find
\begin{align*}
 \lefteqn{-\frac{(n-1)!}{(n-1)!}\ric(P,\nu)\,\vol_N+\frac{1}{(n-1)!}
\hat{\iota}^*(\alpha_n\circ(\call_PB^{ n-1}\wedge1)) }\\
 & \hspace*{30mm} = \hat{\iota}^*\dx(P\lrcorner\alpha_1)+2\binom{n}{2}\langle 
P,\nu\rangle H_2\,\vol_N-\langle P,\nu\rangle\ric(\nu,\nu)\,\vol_N
\end{align*}
or $\dx(\hat{\iota}^*(P\lrcorner\alpha_1))=$
\[ = \frac{1}{(n-1)!}\hat{\iota}^*(\alpha_n\circ(\call_PB^{n-1}\wedge1)) 
+\Bigl(\langle P,\nu\rangle\ric(\nu,\nu)-\ric(P,\nu) -2\binom{n}{2}\langle 
P,\nu\rangle H_2\Bigr)\vol_N, \]
which is the same as identity \eqref{vgHMum} before the integral over the hypersurface $N$ with $\partial N=\emptyset$.

For $i=2$, once we find
$\frac{1}{2(n-2)!}\hat{\iota}^*((P\lrcorner\calri\alpha_n)\circ(B^{n-2}
\wedge1^2))$ and $\hat{\iota}^*(P\lrcorner\calri\alpha_2)$ the proof will be 
almost finished. 
Again we continue from the above formula, 
\eqref{importantcomputation}. Since $\hat{\iota}^*e^0=0$ and $e^0\circ B=0$, 
we have, with $e^q$ in position $p$,
\begin{equation*}
 \begin{split}
\lefteqn{ \hat{\iota}^*((P\lrcorner\calri\alpha_n)\circ(B^{n-2}\wedge1^2)) } 
\\
 &= \hat{\iota}^*\sum_{p,j=0}^n\sum_{q=1}^nR_{p0jq}P_j(e^{1+n}\wedge\cdots 
\wedge e^{p-1+n}\wedge e^{q}\wedge e^{p+1+n} \wedge\cdots\wedge 
e^{2n})\circ(B^{n-2}\wedge1^2) \\ 
&= \hat{\iota}^*\sum_{p,j=0}^n\sum_{q=1}^nR_{p0jq}P_j 
\sum_{\sigma\in\mathrm{Sym}_n} e^{1+n} 
\circ B_{\sigma_1}\wedge\cdots\wedge e^{q}\circ B_{\sigma_p} 
\wedge\cdots\wedge e^{2n}\circ B_{\sigma_n} \\
 &= 2\hat{\iota}^*\sum_{p,j=0}^n\sum_{q=1}^n\sum_{k=1,\,k\neq 
p}^nR_{p0jq}P_j \sum_{\sigma\in\mathrm{Sym}_n:\:\sigma_k=n-1,\, \sigma_p=n} 
e^{1+n}\circ B \wedge\cdots\wedge 
e^{n+k}\wedge\cdots\wedge e^{q}\wedge\cdots\wedge e^{2n}\circ B
  \end{split}
\end{equation*}
where, applying the definition, we have assumed $B_1=\ldots=B_{n-2}=B$ and 
$B_{n-1}=B_n=1$, and where $e^{n+k}$ is in position $k$ and $e^q$ in position 
$p$. Notice the factor 2, which accounts for the permutations such that 
$\sigma_k=n,\,\sigma_p=n-1$. Resuming with the computation,
\begin{equation*}
 \begin{split}
  &= 2(n-2)!\,\hat{\iota}^*\sum_{p,j=0}^n\sum_{q=1}^n\sum_{k=1,\,k\neq p}^n 
 R_{p0jq}P_j\bigl((e^{1}\wedge\cdots\wedge e^{n+k}\wedge\cdots\wedge e^{k} 
\wedge\cdots\wedge e^{n})\delta_{qk} + \bigr.\\
  & \hspace{52mm} \bigl. +(e^{1}\wedge\cdots\wedge e^{n+k}\wedge\cdots\wedge 
e^{p}\wedge\cdots\wedge e^{n})\delta_{qp} \bigr) \\
  &= 2(n-2)!\,\hat{\iota}^*\sum_{p,j=0}^n\sum_{k=1,\,k\neq p}^n  
\bigl(R_{p0jk}P_j e^{1} \wedge\cdots\wedge e^{n+k}\wedge\cdots\wedge 
e^{k}\wedge\cdots\wedge e^{n} + \bigr. \\
 & \hspace{52mm} \bigl. +R_{p0jp}P_je^{1}\wedge\cdots\wedge 
e^{n+k}\wedge\cdots\wedge e^{p}\wedge\cdots\wedge e^{n} \bigr) .
  \end{split}
\end{equation*}
Let $A^k_b=\langle \na_b\nu,e_k\rangle$ denote the matrix of the map 
$\na_\cdot\nu$ in the present frame. Then we recall 
$\hat{\iota}^*e^k=e^k_{|N}$ and $\hat{\iota}^*e^{k+n}=\sum_b A^k_be^b_{|N}$. 
So the formula above yields
\begin{align*}  \label{equacaodePRxialfanedepoisosB}
\frac{1}{2(n-2)!}\hat{\iota}^*((P\lrcorner\calri\alpha_n)\circ(B^{n-2}
\wedge1^2)) & 
 = \sum_{p,j=0}^n\sum_{k=1,\,k\neq p}^n (-R_{p0jk}P_jA^k_p 
+R_{p0jp}P_jA^k_k)\vol_N  \\
 & = \sum_{p,j=0}^n\sum_{k=1}^n(-R_{p0jk}P_jA^k_p+R_{p0jp} P_jA^k_k)\vol_N 
\\  & = -\bigl(\tr{R(P,\na_\cdot\nu)\nu}+nH_1\ric(\nu,P)\bigr)\vol_N.
\end{align*}

Next we compute $\hat{\iota}^*(P\lrcorner\calri\alpha_2)$. First notice 
$e^{jq}\wedge 
e_{p+n}\lrcorner\alpha_2=$
\begin{align*}
  &=e^{jq}\wedge e_{p+n}\lrcorner\frac{2}{2(n-2)!}\sum_{k=1,\,k\neq 
p}^n\ \sum_{\sigma:\,\sigma_k=n-1,\,\sigma_p=n} e^1\wedge\cdots\wedge 
e^{k+n}\wedge\cdots\wedge e^{p+n}\wedge\cdots\wedge e^n   \\
  &=  e^j\wedge\sum_{k=1,\neq p}^n e^1\wedge\cdots\wedge 
e^{k+n}\wedge\cdots\wedge e^{q}\wedge\cdots\wedge e^n  \\
  &=  (1-\delta_{qp}) e^j\wedge e^1\wedge\cdots\wedge e^{q+n} 
\wedge\cdots\wedge e^{q}\wedge\cdots\wedge e^n+  \\ 
 &\hspace{30mm}+\delta_{qp}\sum_{k=1,\neq p}^n e^j\wedge 
e^1\wedge\cdots\wedge e^{k+n}\wedge\cdots\wedge e^{p}\wedge\cdots\wedge e^n
\end{align*}
with $e^{k+n}$ and $e^q$ in position $k$ and $p$ respectively. Hence for 
$j=0$ we have
\begin{align*}
 \lefteqn{\hat{\iota}^*(P\lrcorner( e^{0q}\wedge 
e_{p+n}\lrcorner\alpha_2))=}\\
  &=(1-\delta_{qp})P_0A^q_p e^1\wedge\cdots\wedge e^p\wedge\cdots\wedge 
e^q\wedge\cdots\wedge e^n+  \\ &\hspace{30mm}+
 \delta_{qp}P_0\sum_{k=1,\neq p}^n  A^k_k e^1\wedge\cdots\wedge 
e^k\wedge\cdots\wedge e^{p}\wedge\cdots\wedge e^n \\
  &=  (-(1-\delta_{qp})P_0A^q_p +\delta_{qp}P_0\sum_{k=1,\neq p}^n  
 A^k_k)\vol_N =(-P_0A^q_p +\delta_{qp}nP_0H_1)\vol_N .
\end{align*}
Now we look again the previous computation and its summands when $j=p$ and 
$j\neq p$. Then, for $0<j$, we find that
\begin{align*}
  e^{jq}\wedge e_{p+n}\lrcorner\alpha_2 & =  
\delta_{jp}(1-\delta_{qp}) e^j\wedge e^1\wedge\cdots\wedge 
e^{q+n}\wedge\cdots\wedge e^q\wedge\cdots\wedge e^n+ \\ & \hspace{17mm} 
(1-\delta_{jp}) \delta_{qp} e^j\wedge e^1\wedge\cdots\wedge 
e^{j+n}\wedge\cdots\wedge e^{p} \wedge\cdots\wedge e^n 
\\ &=  \delta_{jp} (1-\delta_{qp}) e^{q+n}\wedge e^1\wedge\cdots\wedge 
e^q\wedge\cdots\wedge e^p\wedge\cdots\wedge e^n-  \\  &
  \hspace{17mm} (1-\delta_{jp}) \delta_{qp} e^{j+n}\wedge 
e^1\wedge\cdots\wedge e^j\wedge\cdots\wedge e^{p}\wedge\cdots\wedge e^n   
\end{align*}
and thus, since $P$ is horizontal,
\begin{align*}
\lefteqn{\hat{\iota}^*(P\lrcorner(e^{jq}\wedge e_{p+n}\lrcorner\alpha_2))=}\\
& =\delta_{jp}(1-\delta_{qp})\bigl(-P_1A^q_1e^{1\cdots n}+ 
P_2A^q_2e^{213\cdots n} -\cdots- (-1)^qP_qA^q_qe^{q12\cdots(q-1) (q+1)\cdots 
n}\cdots\bigr) - \bigl(\,\ldots\,\bigr) \\
 & = \sum_{l=1}^n\bigl(-\delta_{jp}(1-\delta_{qp})P_lA^q_l+
\delta_{qp}(1-\delta_{jp})P_lA^j_l\bigr)\vol_N .
 \end{align*}
In sum, 
\begin{align*}
 \lefteqn{ \hat{\iota}^*(P\lrcorner\calri\alpha_2)  
  =\hat{\iota}^*\bigl(P\lrcorner \sum_{0\leq j< q\leq n} 
 \sum_{p=1}^nR_{p0jq}e^{jq}\wedge e_{p+n}\lrcorner\alpha_2\bigr) } \\
  &=\Bigl(\sum_{q=1}^n nR_{q00q}P_0H_1-\sum_{q,p=1}^nR_{p00q}P_0A^q_p+
  \sum_{1\leq j< q\leq n}\sum_{l=1}^n(R_{q0jq}P_lA^j_l-R_{j0jq}P_lA^q_l)
  \Bigr)\vol_N \\  &= 
\Bigl(-nP_0H_1\ric(\nu,\nu)-P_0\sum_{q,p=1}^nR_{p00q}A^q_p+
  \sum_{j,q=1}^n\sum_{l=1}^nR_{q0jq}P_lA^j_l\Bigr)\vol_N \\
  &= \Bigl(-nP_0H_1\ric(\nu,\nu)-P_0\sum_{q=1}^n\langle 
  R(\nu,\na_p\nu)\nu,e_p\rangle  +\sum_{q=1}^n\langle 
R(\na_P\nu,e_q)\nu,e_q\rangle\Bigr)\vol_N \\
  &= -\bigl(nH_1\langle P,\nu\rangle\ric(\nu,\nu) + \langle P,\nu\rangle 
\tr{R(\nu,\na_\cdot\nu)\nu}+\ric(\na_P\nu,\nu)\bigr)\vol_N .
\end{align*}

Finally we have the identity from the two computations for 
$\hat{\iota}^*(\call_P\alpha_2)$:
\begin{align*} 
 & \frac{1}{2(n-2)!}\hat{\iota}^*((P\lrcorner\calri\alpha_n)\circ(B^{n-2}
\wedge1^2))  + 
\frac{1}{2(n-2)!}\hat{\iota}^*(\alpha_n\circ(\call_PB^{n-2}\wedge1^2))  \\
 & \hspace*{45mm} = \hat{\iota}^*\dx(P\lrcorner\alpha_2) + 3\langle P,\nu\rangle\binom{n}{3}H_3\vol_N 
+ \hat{\iota}^*(P\lrcorner\calri\alpha_2) ,
\end{align*}
which is equivalent to $\dx\hat{\iota}^*(P\lrcorner\alpha_2)=$
\begin{align*} 
 & = -\bigl(\tr{R(P,\na_\cdot\nu)\nu}+nH_1\ric(\nu,P)\bigr)\vol_N  + 
\frac{1}{2(n-2)!}\hat{\iota}^*(\alpha_n\circ(\call_PB^{n-2}\wedge1^2))  \\
 & \ \ \ +\Bigl(nH_1\langle P,\nu\rangle\ric(\nu,\nu) + \langle 
P,\nu\rangle \tr{R(\nu,\na_\cdot\nu)\nu}+\ric(\na_P\nu,\nu)- 3\langle 
P,\nu\rangle\binom{n}{3}H_3\Bigr)\vol_N  ,
\end{align*}
and hence immediately yields formula \eqref{vgHMdois} when $N$ has empty boundary.
\end{proof}

\begin{prop}
  For any vector field $P$ on $M$, lifted as horizontal vector field on $TM$, 
we have
  \begin{equation}
    (\call_PB)Y=B\na^*_YP,\quad\forall Y\in TTM .
  \end{equation}
\end{prop}
\begin{proof}
 $\na^*$ denotes the pullback connection either to $\pi^*TM$ and to  
$\pi^\estrela TM$ thus giving a metric connection over $TM$. The result 
follows immediately from \cite[Proposition 2.1]{Alb2019}, where it is found 
that for any vector field $P$ over $TM$ we have
$(\call_{P}B)Y=B\na^*_YP-\na^*_{BY}P$. In the present setting, $P$ is a 
horizontal lift, hence a pullback to $SM$ to the horizontal tangent 
subbundle. Since $BY$ is vertical, we clearly have $\na^*_{BY}P=0$.
\end{proof}

The case of Theorem \ref{teo_vgHMI} with $P$ Killing-conformal in ambient 
constant sectional curvature $M$ has the first formulations in 
\cite{Katsurada} and more recently in \cite{Kwong}.

We now apply the theorem to a position vector field $P$ on a oriented Riemannian manifold $M$. Let us recall the definition: $P$ is a $\mathrm{C}^1$ vector field over an open subset ${\calu}\subset M$ for which there is a function $f$ such that $\na_XP=f X$, $\forall X\in\XIS_\calu$.

\begin{teo}[Three generalized Hsiung-Minkowski identities]
\label{teo_HMI}
Let $\iota:N\hookrightarrow M$ be a closed oriented immersed hypersurface of 
the oriented Riemannian $(n+1)$-dimensional $\mathrm{C}^2$ manifold $M$. Let 
${P}$ be a position vector field of $M$ defined on a neighborhood of 
$\iota(N)$. Then we 
have that
\begin{equation}\label{HMzero}
 \int_N(f-\langle {P},\nu\rangle H_{1})\,\vol=0
\end{equation}
and
\begin{equation}\label{HMum}
 \int_N\Bigl(2\binom{n}{2}f H_1-2\binom{n}{2}\langle P,\nu\rangle 
H_2-\ric(P,\nu)+\langle P,\nu\rangle\ric(\nu,\nu)\Bigr)\,\vol=0
\end{equation}
and
\begin{equation}\label{HMdois}
 \begin{split}
& \int_N\Bigl(
 \bigl(\langle P,\nu\rangle\ric(\nu,\nu)-\ric(\nu,P)\bigr)nH_1
 +3\binom{n}{3}fH_2-3\binom{n}{3}\langle P,\nu\rangle H_3+ \Bigr. \\
 &\qquad\qquad \Bigl. +\langle 
P,\nu\rangle\tr{(R(\nu,\na_\cdot\nu)\nu)}+\ric(\na_P\nu,\nu)
 -\tr{(R(P,\na_\cdot\nu)\nu)}\Bigr)\,\vol =0.
  \end{split}
\end{equation}
\end{teo}
\begin{proof}
 First, the function $f$ on $M$ corresponds to an obvious function on $TM$ 
constant along the fibres. Then it follows that
$\na^*_{\pi^*Y}\pi^*P=f\pi^*Y$ and, of course, $\na^*_{\pi^\estrela 
Y}\pi^*P=0$, for all $Y\in TM$. Applying the above proposition $\call_{P}B=f 
B$. And we further conclude by a proved Leibniz rule
\begin{equation*}
   \frac{1}{i!(n-i)!}\,\alpha_n\circ(\call_PB^{n-i}\wedge1^i) 
=(n-i)f\,\alpha_i .
\end{equation*} 
Applying $\hat{\iota}^*$ we find $(n-i)f\binom{n}{i}H_i\,\vol_N$, which is 
the same as $(i+1)\binom{n}{i+1}fH_i\,\vol_N$, and the result follows for 
$i=0,1,2$.
\end{proof}
Two short remarks are in order. Identity \eqref{HMzero}, the case $i=0$,
may be attributed to Hsiung in the general Riemannian setting; however,
the meaning of a position vector field in \cite{Hsiung2} is difficult to 
grasp, being given with respect to an assumed orthogonal frame. According to 
\cite{ChenYano,Katsurada}, the position vector is parallel to the mean 
curvature vector field. Then implying $f$ is always 1, which is not even the 
case for space forms. We further remark that, if the hypersurface has 
boundary, the right hand sides of the identities in Theorems \ref{teo_vgHMI} 
and \ref{teo_HMI}, in respective order of $i=0,1,2$, are $\int_{\partial 
N}\hat{\iota}^*(P\lrcorner\alpha_i)$.

The following result gives some useful identities; the proof is trivial.
\begin{prop}
 Let $P$ be a position vector field and let $\nu$ denote a unit vector field 
on $M$. Then we have the following identities:
\begin{equation}
\dx P^\flat=0\qquad \quad 
\call_P\langle\,,\,\rangle=2f\,\langle\,,\,\rangle\qquad \quad R(\;,\;)P=\dx 
f\wedge1   
\end{equation}
\begin{equation}
 \ric(P,\nu)=-n\,\dx f(\nu) \quad \qquad 
\tr(R(P,\na_\cdot\nu)\nu)=nH_1\,\dx f(\nu) .
\end{equation}
In particular, if $f$ is a constant, then $R(\:,\:)P=0$.
\end{prop}

\begin{coro}
 Under the hypothesis of Theorem \ref{teo_HMI}, suppose moreover that $M$ is 
an Einstein manifold. Then we have that
\begin{equation}\label{HM-E-um}
 \int_N\bigl(f H_1-\langle P,\nu\rangle H_2\bigr)\,\vol=0
\end{equation}
and
\begin{equation}\label{HM-E-dois}
 \int_N\Bigl(3\binom{n}{3}fH_2-3\binom{n}{3}\langle P,\nu\rangle H_3+ 
\langle P,\nu\rangle\tr{(R(\nu,\na_\cdot\nu)\nu)}
 -\tr{(R(P,\na_\cdot\nu)\nu)}\Bigr)\,\vol =0.
\end{equation}
\end{coro}
\begin{proof}
 Suppose $\ric$ is a constant multiple of the metric. Then clearly 
$\ric(\na\nu,\nu)=0$. The result follows by simplification of \eqref{HMum} 
and \eqref{HMdois}.
\end{proof}

Thus we may say the \textit{second} Hsiung-Minkowski identity holds true for 
any Einstein metric.

The \textit{third} formula of Hsiung-Minkowski follows immediately for 
constant sectional curvature ambient manifold. Indeed, if we have 
$R(X,Y)Z=c(\langle Y,Z\rangle X-\langle X,Z\rangle Y)$, for all $X,Y,Z\in TM$, for some 
constant $c$, then $\tr{R(P,\na_\cdot\nu)\nu}=-c\langle P,\nu\rangle nH_1= \langle 
P,\nu\rangle\tr{(R(\nu,\na_\cdot\nu)\nu)}$. And therefore, for every 
hypersurface,
\begin{equation}\label{HM-CSC-tres}
  \int_N\bigl(f H_2-\langle P,\nu\rangle H_3\bigr)\,\vol =0.
\end{equation}

Every warped product metric admits a position vector field, cf.
\cite[Corollary 4.1]{Alb2019} or \cite{GuanLi}. Hence the new results in 
this section have a large domain of application.

\vspace{4mm}
\begin{center}
 \textbf{3. Constant mean curvatures $H_1$ and $H_2$ on Einstein $M$}
\end{center}
\setcounter{section}{3}

The applications of Hsiung-Minkowski formulas in space forms are abundant.
We refer the reader to \cite{KwongLeePyo} for an interesting short survey on 
the subject. Now the use of the identities 
(\ref{HMzero},\ref{HM-E-um},\ref{HM-E-dois}) within Einstein manifolds has 
good expectations.

A particular case where the following result applies is that of 
isoparametric hypersurfaces, i.e. all principal curvatures constant.
\begin{teo}
Let $M$ be an Einstein manifold admiting a position vector field $P$. Let $N$ 
be a closed hypersurface of $M$ and suppose $H_1,H_2$ are 
constant. Then:
\begin{enumerate}
 \item[(i)] If $\int_Nf\,\vol=0$ and 
 $\int_N\langle P,\nu\rangle\,\vol\neq0$, then $H_1=H_2=0$ and $N$ is flat;
 \item[(ii)] If $\int_Nf\,\vol\neq0$, then none of the previous vanish 
and $N$ is totally umbilic and non-flat.
\end{enumerate}
\end{teo}
\begin{proof}
 We start by the second case. Let $a_1,\ldots,a_n$ denote the principal 
curvatures of $N$. We have $H_1,H_2$ both non-zero and $H_1/H_0=H_2/H_1$ by 
(\ref{HMzero},\ref{HM-E-um}). In other words, $H_1^2=H_2$ or
 \[  \frac{(a_1+\cdots+a_n)^2}{n^2}=\frac{2}{n(n-1)}\sum_{i<j}a_ia_j . \]
 Equivalently,
 \[ (n-1)(a_1^2+\cdots+a_n^2)-2\sum_{i<j}a_ia_j=0  \]
 or simply 
 \[  \sum_{i<j}(a_i-a_j)^2=0 , \]
 giving the result. For the first case, we have that $H_1=H_2=0$. Then we 
notice the above deductions for case (i) still hold, so we must have all 
$a_i=0$.
 \end{proof}

\vspace{7mm}

The author acknowledges a helpful email conversation with J.~Roth from 
U.~Gustave Eiffel, Paris, regarding some references.

Parts of this article were written while the author was a visitor at Chern 
Institute of Mathematics, Tianjin, P.~R.~China. He warmly thanks CIM for the 
excellent conditions provided and for the opportunity to visit such an
inspiring campus of the U.~Nankai.

\vspace*{8mm}

\textsc{R. Albuquerque}\ \ \ \textbar\ \ \ 
{\texttt{rpa@uevora.pt}}\\
Departamento de Matem\'{a}tica da Universidade de \'{E}vora\\
Centro de Investiga\c c\~ao em Mate\-m\'a\-ti\-ca e Aplica\c c\~oes\\
Rua Rom\~ao Ramalho, 59, 671-7000 \'Evora, Portugal\\
Research leading to these results has received funding from Funda\c c\~ao 
para a Ci\^encia e a Tecnologia (UID/MAT/04674/2013).


\end{document}